\documentclass[11pt]{article}
\usepackage{a4wide}
\usepackage{amsmath}

\usepackage{amsfonts}
\usepackage{amssymb}
\usepackage{euscript}
\newenvironment{proof}{\noindent {\it Proof.~~}\ }{\  \rule{1mm}{2mm}\medskip}

\newenvironment{proofof}[2]{\noindent {\it Proof of #1}~#2: \
}{~\rule{1mm}{2mm}\medskip}
\newtheorem{theorem}{Theorem}
\newtheorem{lemma}[theorem]{Lemma}
\newtheorem{corollary}[theorem]{Corollary}
\newtheorem{proposition}[theorem]{Proposition}

\newcommand{\subgp}[1]{\langle{#1}\rangle}

\begin{document}
\title{
  The global isoperimetric methodology applied to Kneser's Theorem}

\author{ Yahya O. Hamidoune\thanks{Universit\'e Pierre et Marie Curie,    Paris {\tt hamidoune@math.jussieu.fr} }
}

\maketitle

\begin{abstract} We give in the present work a new methodology
that allows to give isoperimetric proofs, for Kneser's Theorem and
 Kemperman's
structure Theory and most sophisticated results of this type. As an
illustration we present a new proof of Kneser's Theorem.
\end{abstract}

\section{Introduction}

A basic tool in Additive Number Theory is the following
generalization of the Cauchy-Davenport Theorem
\cite{cauchy,davenport} due to Kneser:

\begin{theorem}[Kneser \cite{manlivre,natlivre,tv}]\label{kneserl}
Let $G$ be an abelian group and let $A, B\subset G$ be finite
subsets such that $|A+B|\le |A|+|B|-2$. Then $A+B$ is periodic.
\end{theorem}

The above compact form of Kneser's Theorem implies easily the
following popular form of this theorem:

\begin{corollary}[Kneser \cite{manlivre,natlivre,tv}]\label{kneser}
Let $G$ be an abelian group and let $A, B\subset G$ be finite
subsets. Then $|A+B|\ge |A+H|+|B+H|-|H|,$ where $H$ is the period of
$A+B$.
\end{corollary}

Proofs of this result based on  the additive local  transformations
introduced by Cauchy and  Davenport \cite{cauchy,davenport} are
contained in \cite{manlivre,natlivre,tv}.

Recently the author introduced the isoperimetric method allowing to
derive additive inequalities from global properties of the fragments
and atoms (subsets where the objective function $|A+B|-|A|$ achieves
its minimal non trivial value).

This method can be applied to abstract graphs and non abelian groups
and have implications that could not be derived  using the local
transformations. However in the abelian case, it was not clear how
to derive the Kneser-Kemerman's Theory from the isopermetric method.

Very recently Balandraud introduced some isoperimetric objects and
proposed a proof, requiring several pages, of Kneser's Theorem using
as a first step our result that the $1$-atom containing $0$ is a
subgroup.

The purpose of the present paper is not to give a short proof of
Kneser's Theorem. Each of the direct proofs contained in
\cite{manlivre,tv} is quite short and requires around three pages.
However the present proof gives more light on  on the isoperimetric
nature of  Kneser's Theorem and shows that it follows from the
fundamental property of the $1$--atoms.

More interesting than this new proof is the methodology which could
be applied in  the following contexts:
 \begin{itemize}
 \item
It will be used in a coming paper to Kemperman's structure
Theorem \cite{kempacta} and its critical pair Theorem proved
recently by Grynkiewicz \cite{davdecomp}, producing considerable
simplifications.
\item Quite likely this method could be applied to solve the open question  concerning the   description
for subsets $A,B$ with $|A+B|=|A|+|B|+m$, for some values of
$m>0$.
\item This method is purely combinatorial and could be adapted to non
abelian groups. Indeed the major part of the arguments of this
paper holds for non abelian groups.
 \end{itemize}

\section{Terminology and preliminaries}

\subsection{Groups}
Let $G$ denotes an abelian group.  The subgroup generated by a
subset $S$ will be denoted by $\subgp{S}$. Let  $ A,B$ be subsets of
$ G $. The {\em Minkowski sum} is defined as
$$A+B=\{x+y \ : \ x\in A\  \mbox{and}\ y\in
  B\}.$$

For an element $x\in G$, we write $r_{A,B}(x)=|(x-B)\cap A|$. Notice
that $r_{A,B}(x)$ is the number of distinct representations of $x$
as a sum of an element of $A$ and an element of $B$.

We use the following well known fact:
\begin{lemma}\cite{natlivre}
Let $G$ be a finite group and let
 $A,B$  be  subsets
 such that $|A|+|B|\ge |G|+t$.
 Then  $r_{A,B}(x)\ge t$.

\label{prehistorical}
 \end{lemma}

Let  $H$ be a subgroup. A partition  $A=\bigcup \limits_{i\in I}
A_i,$ where $A_i$ is the nonempty intersection of some $H$--coset
with $A$ will be called
 a $H$--{\em decomposition}
of $A$.


\subsection{The strong isoperimetric property}

 Let $V$ be a set and let $E \subset V\times V$.  The relation
$\Gamma = (V,E)$ will be called  a  {\em graph}.
 An element of  $V$
will be called a {\em point} or a {\em vertex}. The graph $\Gamma$
is said to be {\em reflexive} if $(x,x)\in E,$ for all $x$. We shall
write
$$\partial (X)=\Gamma (X)\setminus X.$$

A  {\em  path} of $\Gamma$ from $x_1$ to $x_k$ is a sequence $\mu
=[x_1, \cdots ,x_k]$ of pairwise distinct points (where $k\ge 1$)
such that $(x_i,x_{i+1})\in E$, for all $1\le i \le k-1.$ The set of
points of $\mu$ is by definition $P(\mu)=\{x_1, \cdots, x_k\}$. Our
 paths are called elementary paths in some Graph Theory books.

A family $\mu _1, \cdots , \mu _k$ of paths from $x$ to $y$ will be
called openly disjoint if  $P(\mu _i)\cap P( \mu _j)=\{x,y\}$ for
all $i,j$ with $i\neq j$.

Let $\Gamma =(V,E)$ be a locally finite graph with $|V|\ge 1.$
 The {\em $1$--connectivity}
of $\Gamma$
 is defined  as

\begin{equation}
\kappa _1 (\Gamma )=\min  \{|\partial (X)|\   :  \ \
\infty >|X|\geq 1 \ {\rm and}\ |X \cup  \Gamma(X)|\le |V|-1\},
\label{eqcon}
\end{equation}
where $\min \emptyset =|V|-1$.

Let $G$ be a group, written additively, and let $S$  be a subset of
$G$. The graph $(G,E),$ where $ E=\{ (x,y) : -x+y \ \in S \}$ is
called a {\it Cayley graph}. It will be denoted by $\mbox{Cay}
(G,S)$.

Let $\Gamma =\mbox{Cay} (G,S)$   and  let   $F \subset G $. Clearly
 $\Gamma (F)=F+S $.

A general formalism, including the most recent terminology  of the
isoperimetric method, may be found in a the recent paper
\cite{hiso2007}.

Let $x,y$ be elements of $V$. We shall say that $y$ is {\em
$(k-1)$--nonseparable} to $x$ in $\Gamma$ if $ |\partial (A)|\ge k$,
for every subset $A$ with $x\in A$ and $y\notin \Gamma (A)$.

 We shall formulate  Menger's Theorem (the general form of this result is due to Dirac) which is a basic fact from Graph Theory.
It has applications in Additive number Theory \cite{natlivre,tv}.

\begin{theorem} ( Dirac-Menger)\cite{natlivre,tv} \label{menger}

Let $\Gamma=(V,E)$ be a finite reflexive graph Let $k$ be a
nonnegative integer. Let $x,y\in V$ such that $y$ is
$(k-1)$--nonseparable from $x$, and $(x,y)\notin E$. Then there are
$k$ openly disjoint paths from $x$ to $y$.
\end{theorem}

One may formulate Menger's Theorem for non reflexive graphs. Such a
formulation is slightly more complicated and follows easily from the
reflexive case. We shall give an isoperimetric short proof of this
result in the appendix.

We need the following consequence of Menger's Theorem:

\begin{proposition} { Let $\Gamma $ be a locally finite   reflexive
graph and let $k$ be a nonnegative integer with
 $k\le \kappa _1$. Let   $X$  a
finite subset of $V$ such that $\min (|V|-|X|, |X|)\ge k.$ There are
pairwise distinct elements
 $x_1, x_2, \cdots, x_{k} \in X$ and pairwise distinct elements
 $y_1, y_2, \cdots, y_{k} \notin X$ such that
 \begin{itemize}
   \item $(x_1,y_1), \cdots , (x_{k}, y_{k})\in E,$
    \item $|X\cup \{y_1, \cdots ,  y_{k}\}|=|X|+k$,
 \end{itemize}

\label{strongip}}
\end{proposition}

\begin{proof}
By the definition of $\kappa _1$, we have $|\partial (Y)|\ge \min
(|V|-|Y|,\kappa _1)\ge k ,$ for every $Y\subset V$.
 Let $\Phi=(\Gamma(X),E')$ be
the restriction of $\Gamma$ to $\Gamma (X)$ (observe that $X\subset
\Gamma (X)$). Choose two elements $a,b\notin V.$ Let $\Psi$ be the
reflexive graph obtained by connecting $a$ to $X\cup \{a\}$ and
$\partial (X)\cup \{b\}$ to $b$.  We shall show that $b$ is {\em
$(k-1)$--nonseparable} from $a$ in $\Psi$. Take $a\in T$ such that
$b\notin \Psi (T)$. Then clearly $T\subset X\cup \{a\}. $ Assume
first  $T=\{a\}$. Then $|\Psi (T)|-|T|=|X\cup \{a\}|-1\ge k.$ Assume
now  $T\cap X\neq \emptyset$. We have  $\Psi (T)= X\cup \{a\}\cup
\Gamma (T\cap X)$.  Therefore

\begin{eqnarray*}
 |\Psi (T)|&\ge& 1+|X|
+|\Gamma  (T\cap X) \setminus X|\ge
1+|X|+
(|T\cap X|+k- |X|)>k.
\end{eqnarray*}

By Menger's Theorem there  are $P_1, \cdots , P_{k }$ openly
disjoint paths from $a$ to $b$. Choose $x_i$ as the last point of
the path $P_i$ belonging to $X$ and let $y_i$ the successor of $x_i$
on the path $P_i$. This choice satisfies the requirements of the
proposition.
\end{proof}

We call the property given in Proposition  \ref{strongip}  the {\em
strong isoperimetric property}.

\section{Isoperimetric preliminaries}

The isoperimetric method is usually developed in the context of
graphs. We need in the present work only the special case of Cayley
graphs on abelian groups that we shall identify with group subsets.

Throughout all this section, $S$ denotes a finite generating subset
of an abelian group $G$, with $0\in S$.

 For a subset $X$, we put $\partial _S(X)=(X+S)\setminus X$ and $X^S=G\setminus (X+S)$.

\begin{lemma}\cite{balart,hiso2007}{Let  $X$ be a
 subset of $G$. Then $(X^S)^{-S}+S=X+S$. \label{dualitys}}
\end{lemma}

The last lemma is proved in  Balandraud \cite{balart} and
generalized in \cite{hiso2007}.

 The {\em $1$--connectivity}
of $S$
 is defined  as $\kappa _1 (S )=\kappa _1(\mbox{Cay}(G,S))$. By the
 definitions we have

 \begin{equation}
\kappa _1 (S )=\min  \{|\partial (X)|\   :  \ \
\infty >|X|\geq 1 \ {\rm and}\ |X+S|\le |G|-1\},
\label{eqcon}
\end{equation}
where $\min \emptyset =|G|-1$.

 A finite subset $X$ of $G$ such that $|X|\ge 1$,
$|G\setminus (X+S)|\ge 1$ and $|\partial (X)|=\kappa _1(S)$ is
called a {\em $1$--fragment} of $S$. A $1$--fragment with minimum
cardinality is called a {\em $1$--atom}. The cardinality of a
$1$--atom of $S$  will be denoted by $\alpha_1(S)$.

If $S=G$,   a {\em $1$--fragment} (resp. {\em $1$--atom})
 is just a set with cardinality  $1.$

 These notions, are particular cases  some concepts in
\cite{hcras, Hejcvosp1,halgebra,hactaa,hiso2007}. The reader may
find all basic facts from the isoperimetric method in the recent
paper \cite{hiso2007}.


Notice that
    $\kappa _1 (S)$ is the maximal integer $j$
such that for every finite nonempty subset $X\subset G$

\begin{equation}
|X+S|\geq \min \Big(|G|,|X|+j\Big).
\label{eqisoper0}
\end{equation}

Formulae (\ref{eqisoper0}) is an immediate consequence of the
definitions. We shall call (\ref{eqisoper0}) the {\em isoperimetric
inequality}. The reader may use the conclusion of this lemma as a
definition of $\kappa _1 (S)$. Since  $|\partial (\{0\})|\le \kappa
_1$, we  have:
\begin{equation}\label{bound}
\kappa _1(S)\le |S|-1.
\end{equation}

The basic intersection theorem is the following:

\begin{theorem}\cite{halgebra,hiso2007}

Let  $ S$  be a  generating subset of an abelian  group $G$ with
$0\in S$.
 Let $A$ be a $1$--atom and let
   $F$   be a   $1$-fragment such that  $|A\cap F|\ge 1$. Then
  $A\subset F.$
In particular  distinct $1$-atoms are disjoint.

\label{inter2frag} \end{theorem}

The structure of $1$--atoms is the following:

\begin{proposition} \label{Cay}\cite{hejc2,hjct,hast}

Let  $ S$  be a  generating subset of an abelian  group $G$ with
$0\in S$.  Let $H$ be a $1$--atom of $S$ with $0\in H$.
   Then
   $H$ is a subgroup. 
 Moreover
 \begin{equation}\label{olson}
\kappa _1(S)\geq \frac{|S|}{2},
\end{equation}

\end{proposition}

\begin{proof}
Take $x\in H$. Since $x\in (H+x)\cap H$ and since $H+x$ is a
$1$--atom, we have $H+x=H$ by Theorem \ref{inter2frag}. Therefore
$H$ is a subgroup. Since $S$ generates $G$, we have $|H+S|\ge 2|H|$,
and hence  $\kappa _1(S)=|H+S|-|H|\ge \frac{|S+H|}{2}\ge
\frac{|S|}{2}.$
\end{proof}

Let us formulate two corollaries:

\begin{corollary} \label{Cay}\cite{hejc2,hjct,hast}
Let  $ S$  be a nonempty subset of an abelian  group $G$. Let $Q$ be
the subgroup generated by $S-S$. Let $T$ be a subset of $G$ such
that $T+Q\neq T$. Then

 \begin{equation}\label{olsong}
|T+S|\geq |T|+\frac{|S|}{2},
\end{equation}
\end{corollary}

\begin{proof}
Take an element $a$ of $S$ and put $X=S-a$. Since $X-X =S-S$, $X$
generates $Q$. Take a $Q$--decomposition $T=\bigcup\limits_{i\in J}
T_i$.  Since $T+Q\neq T$, there is a $j$ with $T_j+S\neq T_j$. Take
$b\in T_j$, we have using by (\ref{olson}):
\begin{eqnarray*}
|T+S|&=&\sum\limits_{i\neq j} |T_i+S|+|T_j+S|=|T|-|T_j|+|T_j-b_j+S-a|\\
&=&|T|-|T_j|+(|T_j|+\frac{|S|}{2})=|T|+\frac{|S|}{2}.
\end{eqnarray*}\end{proof}

\begin{corollary} \label{reduct} Let  $ S$ and
$T$ be nonempty subsets of an abelian  group $G$ such that $|T+S|\le
|T|+|S|-m$ and $0\in S$, for some $m\ge 0$.

Then there are $a\in G$ and $T'\subset a+\subgp{S}$, such that
$(T\setminus T')+\subgp{S}=T\setminus T'$ and $|T'+S|\le
|T'|+|S|-m$.

\end{corollary}

\begin{proof}
Decompose $T=\bigcup _{i\in U} T_i$ modulo  $\subgp{S}$. By
(\ref{olson}), $\kappa _1(S)\ge \frac{|S|}2.$ Put $V=\{i\in U :
|T_i+S|<|\subgp{S}|\}.$ By (\ref{eqisoper0}) we have

\begin{eqnarray*}
|T+S| &\ge& \sum \limits_{i\notin V}
|T_i+S|+ \sum \limits_{i\in V}
|T_i+S|\\
 &\ge&(|U|-|V|)|\subgp{S}|+ \sum \limits_{i\in V}
|T_i|+|V|\frac{|S|}2\ge |T|+|V|\frac{|S|}2.
\end{eqnarray*}

It follows that $|V|\le 1.$ The result holds clearly if
$V=\emptyset,$ since $T+S=T+S+\subgp{S}$ in this case. Suppose that
$V=\{\omega\}$. We have clearly $|T_{\omega}+S|\le
|T_{\omega}|+|S|-m$.

\end{proof}

\subsection{Fragments in quotient groups}

We need the following lemma:

\begin{lemma} { Let  $S$ be a
    finite   generating subset of an abelian group $G $  with $0\in S$. Let $H$ be a
subgroup which is a $1$--fragment and let $\phi : G\mapsto G/H$ be
the canonical morphism. Then
\begin{equation}\label{cosetgraph}
\kappa _1(\phi (S))=  |\phi (S)|-1.
\end{equation}
}\label{quotient}
\end{lemma}
\begin{proof}

 Put $|\phi (S)|=u+1$.  Since $|G|>|H+S|,$ we have $\phi (S)\ne G/H$,
 and hence $\phi (S)$ is
$1$--separable.

Let $X\subset G/H,$ be such that  $X+\phi (S)\neq G/H$. Clearly
$\phi^{-1} (X)+S\neq G$. Then $|\phi^{-1} (X)+S|\ge |\phi^{-1}
(X)|+\kappa _1(S)= |\phi^{-1} (X)|+u|H|.$

It follows that $|X+\phi (S)||H|\ge |X||H|+u|H|.$ Hence $\kappa
_1(\phi (S))\ge u=|\phi (S)|-1$.\end{proof}

\section{An isoperimetric proof of Kneser's Theorem}

\begin{proofof}{Theorem}{\ref{kneserl}}

 Without loss of generality we may assume   that $0\in S$ and $|S|\le |T|$.
The proof is by induction on $|S|+|T|$, the result being obvious for
$|S|+|T|$ small.

{\bf Claim 1} If $T\not\subset \subgp{S}$, then the result
holds.

\begin{proof} By Corollary \ref{reduct},
 there are $a\in G$ and $T'\subset a+\subgp{S}$, such that
$(T\setminus T')+\subgp{S}=T\setminus T'$. and $|T'+S|\le
|T'|+|S|-2$. Without loss of  generality we may assume that $0\in
T'$.
 By the
induction hypothesis there is a non zero subgroup $N$ of
$\subgp{S}$, such that $T'+S+N=T'+S$. It follows that $T+S+N=T+S$.
\end{proof}

By Claim 1, we may assume without loss
  of generality that
  $$G=\subgp{S}.$$

Assume first $|G|-|T+S|=|T^S|< |T|$. Then $G$ is finite. By the
definition $(T^S-S)\cap T=\emptyset$. Therefore $|T^S-S|\le
|G|-|T|=|G|-|S+T|+|S+T|-|T|\le |T^S|+|S|-2$. Since
$|T^S|+|S|<|T|+|S|$, we have by the induction hypothesis,
$T^S-S=T^S-S+N$, for some non zero subgroup $N$. Then $(G\setminus
(T^S-S))={T^S}^{-S}$ is $N$--periodic, and hence by Lemma
\ref{dualitys}  $T+S={T^S}^{-S}+S$ is $N$--periodic. So we may
assume

\begin{equation}\label{y<} |S|\le |T|\le
|T^S|.
\end{equation}

We prove first the bound

\begin{equation}\label{eq2n/3}
|S+T|\le   \frac{2|G|-2}{3}.
\end{equation}

By the assumption $|T^S|=|G|-|T+S|\ge |T|\ge |S|$, we have

 \begin{eqnarray*}
 3|S+T|&\le &2|S+T|+|S|+|T|-2\\
 &\le&|G|-|S|+|G|-|T|+|S|+|T|-2= 2|G|-2,
  \end{eqnarray*}
which proves (\ref{eq2n/3}).

Let $H$ be a $1$-atom $S$ and let $\phi :G\mapsto G/H$ denotes the
canonical morphism. Put $|\phi(S)|=u+1$ and $|\phi(T)|=t+1$.

Take a $H$--decomposition $S=\bigcup \limits_{0\le i\le u}S_i$ such
that $|S_0|\ge \cdots \ge |S_u|$. By the definition of a $1$-atom we
have $u|H|=|H+S|-|H|=\kappa _1\le |S|-2.$ It follows that for all
$u\ge j\geq 0$
\begin{equation}\label{plein}
|S_{u-j}|+\cdots +|S_u|\ge j|H|+2
\end{equation}

It follows that $|S_0| \ge \frac{|H|+2}{2}$. In particular $S_0$
generates $H$.

We shall use this fact in the application of the isoperimetric
inequality.

 Take a $H$--decomposition $T=\bigcup \limits_{0\le i\le t}T_i$.


 By (\ref{cosetgraph}), $\kappa _1
(\phi(S))=|\phi (S)|-1=u.$ Put $\ell =\min (q-t-1,u)$.

 By Proposition \ref{strongip} applied to  $\phi (S)$ and $\phi (T)$, there is a subset $J\subset
[0,t]$ with cardinality $\ell$ and a family $\{ mi ;i\in J\}$ of
integers in $[1,u]$
  such that  $T+S $ contains the $H$--decomposition $(\bigcup \limits_{0\le i\le
  t}T_i+S_0)\cup (\bigcup \limits_{ i\in J}T_i+S_{mi})\cup R$,

where $R=(S+T)\setminus  ((\bigcup _{i\in J} {T_i+S_{mi}}+H)\cup
(\bigcup _{0\le i\le t} T_i+H))$.

We shall choose such a $J$ in order to maximize $|J\cap P|.$ We
shall write $E_i=(S+T)\cap (T_i+H)$, for every $i\in [0,t]$. Also we
write $E_{mi}=(S+T)\cap (T_i+S_{mi}+H)$, for every $i\in J$.

 We put also  $W=\{i \in [0,t] :  |E_i|<|H|\},$ and $P=[0,t]\setminus W.$
We write also  $q=\frac{|G|}{|H|}$.

Since $|T|\ge |S|$ we have $|T+H|\ge |S|>\kappa_1(S)=u|H|.$ It
follows that $t+1=|\phi (T)|\ge u+1.$ Then $t+1-|J|>0.$ In
particular $I\neq \emptyset ,$ where $I=[0,t]\setminus J.$

Let $X$ be a subset of $I$ and  let $Y$ be a subset of $J$. We have
\begin{eqnarray}
|S+T|-|R|
&\ge &\sum \limits_{i\in X\cup Y}|E_i|+ \sum \limits_{i\in I\setminus X\cup J\setminus Y}|T_i+S_{0}|
+\sum\limits _{i\in J\setminus Y}|T_i+S_{mi}|+\sum \limits_{i\in Y}|E_{mi}|\nonumber\\
&\ge &\sum \limits_{i\in X\cup Y}|E_i|+ \sum \limits_{i\in I\setminus X \cup J\setminus Y} |T_i| + (u-|Y|)|S_{0}|+\sum \limits_{i\in Y}|E_{mi}|\label{debut}\\
&\ge &\sum \limits_{i\in X\cup Y}|E_i|+ \sum \limits_{i\in I\setminus X\cup J\setminus Y} |T_i|+  (u-|Y|)|S_{0}|
+|Y||S_{u}|\label{eqdebut}
\end{eqnarray}

Put $F=\{i\in I\cap P : (T_i+S) \cap (\bigcup _{i\in W}T_i+H)\neq
 \emptyset\}$.

We shall use the following obvious facts: For all $i\in W$, we have
by (\ref{olson}), $|E_i|\ge|T_i+S_0|\ge |T_i|+\kappa _1(S_0)\ge
|T_i|+\frac{|S_0|}{2}.$ For every $i\in F$, $T_i+S_{ri} \subset
T_j+H$ for some $1\le ri\le u$ and some $j\in W.$ Hence we have
$|T_i|+|S_{u}|\le |T_i|+|S_{ri}|\le |H|=|E_i|,$ by Lemma
\ref{prehistorical}.

Let $U$ be a subset of $W\cap J$. Put $X=I$ and $Y=U$. By
(\ref{eqdebut}), we have

\begin{eqnarray}
|S+T|-|R|
&\ge &\sum \limits_{i\in U\cup (W\cap I)\cup (P\cap I)}|E_i|+ \sum \limits_{i\in J\setminus U} |T_i|+ (u-|U|)|S_{0}|+|U||S_{u}|
\\
&\ge &\sum \limits_{i\in (P\cap I)\setminus F}|T_i|+\sum \limits_{i\in F}(|T_i|+|S_u|)+ \sum \limits_{i\in (W\cap I) \cup U}(|T_i|+\frac{|S_{0}|}2)
+|J\setminus U||S_{0}|
+|U||S_{u}|\nonumber\\
&\ge& |T|+  |J\setminus U||S_0|+(|U|+|F|)|S_u|+|(W\cap I) \cup U|\frac{|S_{0}|}2.\label{eqvide}
\end{eqnarray}

{\bf Claim 2} $q \ge |\phi (S)|+|\phi (T)|-1$, and hence $\ell =u$.

\begin{proof} The proof is by contradiction. Suppose that $q < |\phi (S)|+|\phi
(T)|-1$.

 Assume first $u\ge 2$. By Lemma \ref{prehistorical}, the
are two distinct values of the pair $(s,t)$ such that $T_s+S_t
\subset E_{mi}$, for every $i\in J$. In particular $|E_{mi}|\ge
|S_{u-1}|$, for every $i\in J$. Also $|E_{i}|\ge |S_{0}|$, for every
$i\in [0,t]$.

  Observe that $2t> t+u\ge q$. We have using (\ref{plein})

  $ 2|S_0|\ge |S_0|+|S_{u-1}|\ge \frac{2}{3}(|S_u|+|S_{u-1}|+|S_{u-2})>\frac{4|H|}3$.
  By (\ref{eqdebut}), applied with $X=I$ and $Y=J$, we have

\begin{eqnarray*}
|S+T|&\ge& \sum \limits_{0\le i \le t} |S_0|+ \sum \limits_{i\in J} |S_{u-1}|= (t+1)|S_0|+(q-t-1)|S_{u-1}|\\
&=& (2t+2-q)|S_0|+(q-t-1)(|S_0|+|S_{u-1}|)\\
&>& (2t+2-q)\frac{2|H|}3+\frac{4|H|(q-t-1)}3=\frac{2|G|}3,
\end{eqnarray*}

contradicting (\ref{eq2n/3}).

Assume now $u=1.$

From the inequality $|T+S|\le |T|+|S|-2$, we see that $\kappa
_1(S)\le |S|-2$. Therefore we have by (\ref{eq2n/3}),
$\frac{2|G|}{3}>|T+S|\ge |T|+\kappa _1(S)\ge |S|+|H|>2|H|$, and
hence
$$q\ge 4.$$

 We have $(t+1)+(u+1)-1<|\phi (S+T)|\le  q.$ Then $t+1=q.$
Hence $\ell =|J|=0$.
 We have $|W|\geq 1,$ since
otherwise $G= T+H\subset S+T$. We have $|W|\leq 3,$ by
(\ref{eqvide}) applied with $U=\emptyset$. Therefore $|P|\geq
t+1-3\ge 4-3=1$. There is clearly   $i\in P$ with $T_i+S_1 \subset
T_j+H$ for some $j\in W,$ and hence $|F|\ge 1$. By (\ref{eqvide})
applied with $U=\emptyset$, $|T+S|\ge |T|+|W|\frac{|S_0|}{2}+|S_1|$,
and hence $|W|\le 1$. It follows that $|S+T|\ge |G|-|H|=
|G|-\frac{|G|}q \ge \frac{3|G|}{4}$, contradicting
(\ref{eq2n/3}).\end{proof}

We must have $R = \emptyset$, since otherwise by (\ref{eqvide})
applied with $U=\emptyset$, $|S+T|-|R|\ge |S+T|-|S_u||\phi(R)|\ge
|T|+u|S_0|+|S_u|\ge |T|+|S|,$ a contradiction. In particular

\begin{equation} \label{cdcp}
|\phi (S+T)| = |\phi (S)|+|\phi (T)|-1.
\end{equation}

 {\bf Claim 3}.  $J \cap P \neq \emptyset.$

\begin{proof}
Suppose the contrary and take $k\in J\cap W$. Put $U=\{k\}$. By
(\ref{eqvide}),

\begin{eqnarray*}
|S|+|T|>|S+T|&\ge& |T|+ (u-1)|S_{0}|+|S_{u}|+(|W\cap I|+1)\frac{|S_{0}|}2.
\end{eqnarray*}

It follows that $I\subset P$. Since $S$ generates $G$, we have
$|\bigcup _{i\in I} T_i+H+S|>|\bigcup _{i\in I} T_i+H|$.

 We must have $ (\bigcup _{i\in I} T_i+H+S)\cap (\bigcup _{i\in J} E_{mi}+H)=\emptyset $,
 since otherwise by replacing a suitable element of $J$ with some $p\in I$, we
may increase strictly $ |J\cap P|,$ observing that $I\subset P$.

By (\ref{cdcp}), there are $i\in I$, $j\in J$ and $p\in [1,u]$ such
that $ T_i +S_p$ is congruent $T_j+S_{mj}$. It follows that $F\neq
\emptyset$.

By (\ref{eqvide}) applied with $U=\emptyset$,

\begin{eqnarray*}
|S+T|&\ge& |T|+ u|S_{0}|+|S_{u}|\ge |T|+|S|,
\end{eqnarray*}
a contradiction proving the claim.
\end{proof}

Take $r\in  J$ with $|E_r|=|H|$. Such an $r$ exists by Claim 3.

 {\bf Claim 4} $T_i+H+S_j=T_i+S_j,$
for all $0\le j\le u-1$.

 \begin{proof} By Lemma \ref{prehistorical}, it would be enough to show the
following:
\begin{equation} \label{FINAL}
|T_k|+|S_{u-1}|>|H|,
\end{equation}
for every $k\in [0,t]$.
 Suppose the contrary.

Notice that $|E_{mr}|\ge \max (|T_r|,|S_u|)$ and that $|E_k|\ge
|S_0|$. Also $|T_k|+|S_{u-1}|\le |H|=|E_{mr}|$ by our hypothesis. We
shall use  these inequalities and (\ref{debut}) with $X=\{k,r\}\cap
I$ and $Y=\{k,r\}\cap J$.

 By (\ref{debut}) we have for for $k\neq r$,
\begin{eqnarray*} |S+T|
&\ge&  |T|-|T_k|-|T_r|+ (u-|X|)|S_0|+|T_k|+|S_{u-1}|+|S_{0}|+|T_r|+|Y||S_{u}|\\&\ge&  |T|+(u-1) |S_0|+|S_{u-1}|+|S_{u}|\ge |T|+|S|,
\end{eqnarray*}
leading a contradiction. If $k=r$ the contradiction comes more
easily.
\end{proof}

 Put $D=\{i\in J :
T_i+S_{mi}+H\not\subset S+T  \}$ and  $C=(T+H)\cup \bigcup
\limits_{i\in J\setminus D} (T_i+S_{mi}+H)\cup \bigcup \limits_{i\in
D} (T_i+S_{u}) $.

 We shall
show that \begin{equation}\label{FF} T+S= C. \end{equation}
 By
(\ref{cdcp}), we have $|\phi(C)|=t+1+u=|\phi(S+T)|$.  By the
definition of $D$ and by Claim 4, $C\setminus (\bigcup \limits_{i\in
D} (T_i+S_{u}))$ is $H$--periodic subset of $S+T$. It remains to
show that the traces of $S+T$ and $C$ coincide on the  cosets
represented by elements in $\bigcup \limits_{i\in D} (T_i+S_{u}).$
Take $i\in D$. It follows by Claim 4 that $mi=u.$ We can not have
$T_l+S_j \equiv T_i+S_u,$ mod $H$ for some $j\neq u$, since
otherwise by Claim 4 $T_l+H+S_j=T_l+S_j,$ and $i\notin D$, a
contradiction. The proof of (\ref{FF}) is compete.

 Let $Q=\subgp{S_u-S_u}$. By
(\ref{plein}) we have $|Q|\ge |S_u|\ge 2.$ Put $D'=\{i\in D :
T_i+S_{mi}+Q\neq T_i+S_{mi} \}$. By (\ref{olsong}) we have,
$|T_i+S_u|\geq |T_i|+\frac{|S_u|}{2}.$

By  the definition of $D'$  and since $Q\subset H$, we have

have using (\ref{debut}), applied with $X=\emptyset $ and $Y=D'$
\begin{eqnarray*} |S+T|
&\ge& |T|-\sum \limits_{ i\in D'} |T_{i}|+\sum \limits_{ i\in D'} |E_{mi}|+u|S_0| \ge  |T|+ u|S_0|+|D'|\frac{|S_{u}|}2.
\end{eqnarray*}

Clearly $T+S+Q=T+S$ if $D'=\emptyset$. Suppose $D'\neq \emptyset$.
We must have $|D'|\le 1,$ since otherwise $|S+T|\ge
|T|+u|H|+|S_u|\ge |T|+|S|,$ a contradiction. Then $|D'|=1$. Put
$D'=\{o\}$. Take $x_o\in T_o$.We have
$|T_o-x_0+S_u-a_u|=|T_o+S_u|\leq |T_o|+|S_u|-2$ since otherwise
$|S+T|\ge |S|+|T|-1$. By the induction hypothesis there is a nonzero
subgroup $N$ of $Q$, with $T_o-x_o+s_u-a_u+N=T_o-x_o+s_u-a_u$. It
follows that $T_o+S_u+N=T_o+S_u$. We have clearly $S+T+N=S+T$.

\end{proofof}

\section{Appendix : An isoperimetric proof of  Menger's Theorem}

We present here an isoperimetric proof of Menger's Theorem. Let $E
\subset V\times V$ and let $\Gamma=(V,E)$ be a reflexive graph. For
a subset $X$ of $V$, we put $X^\curlywedge=V\setminus { \Gamma
(X)}$. Let $x,y$ be elements of $V$.
 The graph $\Gamma$ will be called {\em
$(x,y)$--$k$--critical} if $y$ is {\em $(k-1)$--nonseparable} from
$x$ in $\Gamma$, and if this property is destroyed by the deletion
of every arc $(u,v)$ with $u\neq v$.

A subset $A$ with $x\in A$ and $y\notin \Gamma (A)$ and $|\partial
(A)|=k$ will be called a {\em $k$--part} with respect to
$(x,y;\Gamma)$.

 The reference to $(x,y)$ will be omitted.

\begin{lemma}\label{frxxy}
Assume that $\Gamma=(V,E)$ is {$k$--critical} and let  $(u,v)\in E$
be an arc with $u\neq v$. Then $\Gamma$ has $k$--part $F$ with $u\in
F$ and $v\in
\partial (F)$.
\end{lemma}

\begin{proof}

Consider the graph $\Psi=(V,E \setminus\{(u,v)\})$.
 There is  an $ F$ with $x\in F$ and $y\notin \Psi (F)$ such that $|\partial _{\Psi}(F)|<k$.
 This forces that  $u\in F$ and that $vin \partial (F)$, since otherwise $\partial _{\Psi} (F)=\partial
_{\Gamma}(F)$.

Since $\partial _{\Psi} (F)\cup \{v\} \supset \partial
_{\Gamma}(F)$,
 we have  $|\partial _{\Gamma}(F)|\le k$. We must have
 $|\partial _{\Gamma }(F)|= k$, since $y$ is {\em $(k-1)$--nonseparable} from $x$
in $\Gamma$. This shows that $F$ is a $k$-part.

\end{proof}

\begin{lemma} \label{dualxy}
Let $F$ be a $k$--part with respect to $(x,y;\Gamma)$. Then
$F^\curlywedge$ is a $k$--part a with respect to
$(y,x;\Gamma^{-1})$. Moreover $\partial _{-}  (X^\curlywedge)=
\partial (X).$

 In particular   $x$ is
 $(k-1)$--nonseparable from $y$ in  $\Gamma^{-1}$,
  if $y$ is {$(k-1)$--nonseparable} from $x$  in
$\Gamma$.
\end{lemma}
\begin{proof}

We have clearly
 $\partial _{-}  (X^\curlywedge)\subset
\partial (X).$ Put
 $C=\partial (X)\setminus \partial _{-}  (X^\curlywedge).$

 Since $y\notin \Gamma (X\cup C)$,  we have $k\le |\partial (X\cup
C)|\le |\partial _{-} (X^\curlywedge)|\le |\partial (X)|=k$.

\end{proof}

The above lemma is a local version of the isoperimetric duality.

\begin{lemma} \label{frsxy}
Assume that $\Gamma=(V,E)$ is {$k$--critical} and that $\Gamma
(x)\cap \Gamma ^{-1}(y)=\emptyset$. There is a $k$--part $F$ of
$\Gamma$ such that $\min (|F|,|F^\curlywedge|)\ge 2.$
\end{lemma}
\begin{proof}

Take a path  $[x, a,b, \cdots,c,y]$ of minimal length from $x$ to
$y$. By Lemma \ref{frxxy}, there is a $k$-part $F$, with $a\in F$
and $b\in
\partial (F)$. We have $\{x,a\}\subset F$. We have
$|F^\curlywedge|\ge 2$ since otherwise $F^\curlywedge=\{y\}$. Hence
by Lemma \ref{dualxy}, $b \in \partial (F)=\partial ^-(\{y\})$.
Therefore $b\in \Gamma (x)\cap \Gamma ^{-1}(y),$ a contradiction.

\end{proof}

Let $x$ be an element of $V$ and let $T=\{y_1,\cdots ,y_k\}$ be a
subset of $V\setminus \{v\}$. A family of $k$--openly disjoint paths
$P_1, \cdots , P_k$, where $P_i$ is a path  from $x$ to $y_i$ will
be called an {\em $(x,T)$--fan}.

\begin{proofof}{Theorem}{\ref{menger}}

 The proof is by induction, the result being obvious for
$|V|$ small. Assume first that there $z\in \Gamma (x)\cap \Gamma
^{-1}(y)$. Consider the restriction $\Psi$ of $\Gamma$ to
$V\setminus \{z\}$.  Clearly $y$ is {$(k-2)$--nonseparable} from $x$
in $\Psi$. By the induction hypothesis  there are $(k-1)$--openly
disjoint paths from $x$ to $y$ in $\Psi$. We adjoin the path
$[x,z,y]$ to these paths and we are done. So we may assume that
$\Gamma (x)\cap \Gamma ^{-1}(y)=\emptyset$.

By Lemma \ref{frsxy} there is a part $F$ with  $\min
(|F|,|F^\curlywedge|)\ge 2.$ Consider the reflexive graph
$\Theta=(V',E')$ obtained by contracting $F^\curlywedge$ to a single
vertex $y_0$. We have $V'=(V\setminus F^\curlywedge)\cup \{y_0\}$.
Since $|V'|<|V|$, by the induction hypothesis there are $k$ openly
disjoint paths form $x$ yo $y_0$. By deleting $y_0$ we obtain an
$(x, \partial (F))$--fan. Similarly by contracting $F$ and applying
induction, we form a $( \partial (F),y)$--fan.

By composing these two fans, we form  $k$ openly disjoint paths from
$x$ to $y$.
\end{proofof}


\bigskip

{\bf Remerciement} {\small L'auteur tient \`a remercier Jean Claude
Fournier pour des discussions stimulantes au tour du Th\'eor\`eme de
Menger.

 l'Auteur tient aussi \`a remercier
Eric balandraud, David Grynkiewicz, Alain Plagne, Oriol Serra et
Gilles Z\'emor pour d' autres discussions sur des probl\`emes
additifs et isop\' erim\'triques.

Ce manuscrit a \'et\'e pr\'epar\'e lors d'un s\'ejour en Mauritanie
o\`u l'auteur a pu b\'en\'eficier d'un climat favorable au travail
dans les locaux du CFED à Nouakchott. Je tiens donc à remercier M.
Isselmou O. Mohamed et l'ensemble son personnel.}


\begin{thebibliography}{99}



\bibitem{balart} E. Balandraud, Un nouveau point de vue isop\'erimetrique
appliqu\'e au th\'eor\`eme de Kneser, {\it Preprint}, december 2005.




\bibitem{cauchy} A.  Cauchy,  Recherches  sur  les  nombres,  {\it J.  Ecole  polytechnique}
9(1813), 99-116.



\bibitem{davenport} H. Davenport, On the addition of residue  classes, {\it J.  London  Math.
Soc.} 10(1935), 30--32.





\bibitem{dixmier}  J. Dixmier, Proof of a conjecture
by Erd\"os, Graham concerning the problem of Frobenius, {\it J.
number Theory} 34 (1990), 198-209.






\bibitem{davdecomp} D. Grynkiewicz, Quasi-periodic decompositions and the Kemperman's structure
theorem,  European J. Combin.  26  (2005),  no. 5, 559--575.

\bibitem{davkem} D. Grynkiewicz, A step beyond Kemperman's structure Theorem, Preprint May 2006.


\bibitem{hcras} Y.O. Hamidoune, Sur les atomes d'un graphe orient\'e,
{\it C.R. Acad. Sc. Paris A}  284 (1977),   1253--1256.


\bibitem{hjct} Y.O. Hamidoune,   Quelques probl\`emes de connexit\'e dans les
graphes orient\'es,  {\it J. Comb. Theory} B 30 (1981), 1-10.


\bibitem{hejc2} Y.O. Hamidoune, On the connectivity of Cayley digraphs,
{\it Europ. J.  Combinatorics}, 5 (1984), 309-312.

\bibitem{Hejcvosp1} Y.O. Hamidoune, Subsets with small sums in abelian groups I: The Vosper property.
European J. Combin. 18 (1997), no. 5, 541--556.


\bibitem{halgebra} Y.O. Hamidoune, An isoperimetric method in additive theory.
{\it J. Algebra} 179 (1996), no. 2, 622--630.


 \bibitem{hast} Y.O. Hamidoune,  On small subset product in a group.
Structure Theory of set-addition,  {\it Ast\'erisque}  no.
258(1999), xiv-xv, 281--308.

 \bibitem{hactaa} Y.O. Hamidoune, {Some results in Additive number
Theory I: The critical pair Theory}, Acta Arith. 96, no. 2(2000),
97-119.



\bibitem{hiso2007} Y.O. Hamidoune, Some additive applications of the isopermetric approach,
 http://arxiv.org/abs/math./07060635.


\bibitem{hplagne1} Y. O. Hamidoune , A. Plagne.
{A new critical pair theorem applied to sum-free sets}.
  {\it Comment. Math. Helv.} 79 (2004), no. 1, 183--207.















\bibitem{kempacta} J. H. B. Kemperman, On small sumsets in Abelian groups,
{\it Acta Math.} 103 (1960), 66--88.




\bibitem{levkemp} V. F. Lev,  Critical pairs in abelian groups and Kemperman's structure theorem.
{\it Int. J. Number Theory} 2 (2006), no. 3, 379--396.

\bibitem{manlivre} H.B. Mann, {\it Addition Theorems},   R.E.
Krieger, New York, 1976.
\bibitem{natlivre}M. B. Nathanson,
{\it Additive Number Theory. Inverse problems and the geometry of
sumsets}, Grad. Texts in Math. 165, Springer, 1996.







\bibitem{tv} T. Tao, V.H. Vu,  {\it Additive Combinatorics}, Cambridge Studies
in Advanced Mathematics 105 (2006), Cambridge University Press.











\end{thebibliography}
\end{document}